\documentclass[12pt]{amsart}

\usepackage{graphicx, amssymb, amsthm, amsfonts, latexsym, epsfig, amsmath}
\usepackage{verbatim}
\usepackage{pstricks}
\usepackage{setspace}
\usepackage{pstricks-add, pst-math,pst-xkey}
\newtheorem{thm}{Theorem}[section]
\newtheorem{lem}[thm]{Lemma}

\newtheorem{prop}[thm]{Proposition}
\newtheorem{col}[thm]{Corollary}
\theoremstyle{definition}
\newtheorem{defn}[thm]{Definition}

\newtheorem{rem}[thm]{Remark}

\newcommand{\nempty}{\not=\emptyset}

\newcommand{\B}[1]{\ensuremath{\mathbb{#1}}}
\newcommand{\N}{\B{N}}

\newcommand{\C}[1]{\mathcal{#1}}

\newcommand{\al}{\alpha}

\renewcommand{\include}{\input}

\newcommand{\w}{\omega}

\newcommand{\sh}{\sigma}

\newcommand{\eps}{\varepsilon}

\newcommand{\nat}{\mathbb{N}}

\newcommand{\restr}[2]{#1\hspace{-0.15cm}\restriction_{#2}}

\newcommand{\rtr}{\hspace{-0.15cm}\restriction}
\newcommand{\app}{\approx}

\title{The $\omega$-limit sets of quadratic Julia sets}

\author[A. D. Barwell]{Andrew D. Barwell}
\address[A. D. Barwell]{School of Mathematics, University of Bristol, Howard House, Queens Avenue, Bristol, BS8 1SN, UK -- and -- School of Mathematics, University of Birmingham, Birmingham, B15 2TT, UK}
\email[A. D. Barwell]{A.Barwell@bristol.ac.uk}
\author[B. E. Raines]{Brian E. Raines}
\address[B. E. Raines]{Department of Mathematics, Baylor University, Waco, TX 76798--7328,USA}
\email[B. E. Raines]{brian\_raines@baylor.edu}

\begin{document}

\begin{abstract}
  In this paper we characterize $\w$-limit sets of dendritic Julia sets for quadratic maps.  We use Baldwin's symbolic representation of these spaces as a non-Hausdorff itinerary space and prove that quadratic maps with dendritic Julia sets have shadowing, and also that for all such maps, a closed invariant set is an $\w$-limit set of a point if, and only if, it is internally chain transitive.
\end{abstract}

\maketitle

\section{Introduction}

In this paper we consider the dynamics of quadratic maps on the complex plane, $f_c(z)=z^2+c$.  The interesting dynamics of such a map are carried on the Julia set which is often a strange self-similar space.  There are many values of $c$ for which the Julia set is a \emph{dendrite},  a locally connected and uniquely arcwise connected compact metric space.  One example is when the value $c$ is strictly pre-periodic; such values are known as Misiurewicz points and the corresponding Misiurewicz maps are well studied (see for example \cite{AlsedaFagella, D-H, Schleicher}).

Baldwin gives an efficient encoding of the dynamics of these quadratic maps restricted to the Julia set, into a shift map on a non-Hausdorff itinerary space \cite{Baldwin-JFPTA}.  A detailed description of all such itinerary spaces is given in \cite{Baldwin-TopApp}, and an illustration of the utility of non-Hausdorff itinerary spaces in analyzing dynamical systems is given in \cite{Baldwin-Fund}.  Baldwin's collection of itinerary spaces is a large family of dendrites and maps that includes all of the dendritic Julia sets for quadratic maps.  It also includes maps on self-similar dendrites which are not Julia sets.  Our $\w$-limit set characterization theorem applies to both the Julia set dendrites and the dendrite maps which are not Julia sets.

\medskip
A set $\Lambda$ is \emph{internally chain transitive (ICT)} provided for every $\eps>0$ and every pair $x,y\in \Lambda$ there is an \emph{$\eps$ pseudo-orbit} from $x$ to $y$ in $\Lambda$; in other words a collection $\{x=x_0, x_1, \dots x_n=y\}\subseteq \Lambda$ such that $d(f(x_i), x_{i+1})<\eps$ for $0\le i\le n-1$.  In the family of compact metric spaces this property is equivalent with Sarkovskii's property of weak incompressibility, \cite{BGOR-preprint}.  We say a set $M$ is \emph{weakly incompressible} provided for every nonempty closed proper subset $K$ of $M$ we have \[K\cap \overline{f(M-K)}\nempty.\]  Sarkovskii proved that every $\w$-limit set has weak incompressibility, and so every $\w$-limit set is ICT (a fact also proved in \cite{Hirsch}). In this paper we consider the question of the converse: \emph{`Is every ICT set an $\w$-limit set?'}

Let $f:X\rightarrow X$ on a metric space $(X,d)$ be continuous.  A \emph{$\delta$ pseudo-orbit} is a finite or infinite ordered set $\{x_i\}$  such that $d(f(x_i), x_{i+1})<\delta$ for all $i$.  We say the map $f$ has \emph{shadowing} if for every $\eps>0$ there is a $\delta>0$ such that for every $\delta$ pseudo-orbit $\{x_i\}$ there is a point $z\in X$ for which $d(f^i(z),x_i)<\eps$ for all $i$; we say that $z$ \emph{$\eps$-shadows} the pseudo-orbit. In \cite{BDG-tentmaps} the authors conjecture that the answer to the question \emph{`Is every ICT set an $\w$-limit set?'} is `yes' for any map which has shadowing. The conjecture is supported by a result that shows for every tent map with a periodic critical point (maps which have shadowing) the answer is `yes', and by an example showing there are tent maps without shadowing for which the answer is `no'. Furthermore, the authors prove in \cite{BGKR-omega} that the answer is `yes' in the class of shifts of finite type (SFTs), a class of map which has shadowing \cite{Walters}, but the answer is `no' in the class of sofic shifts (which do not have shadowing in general).  There are maps of the unit interval for which the answer is `yes' and others for which the answer is `no', but for which shadowing is neither proved nor disproved \cite{Barwell-Fund}.  Through the use of Baldwin's encoding, we demonstrate that the answer is `yes' for the class of dendritic Julia sets of quadratic maps, and that these maps have shadowing, further supporting the above conjecture.  In fact we prove these results for all self-similar dendrite maps with the unique itinerary property --  a large class of maps containing the dendritic Julia set maps. 

\medskip
The paper is organized as follows.  In the next section we present the preliminary definitions and results for dendrite maps.  In Section Three we prove lemmas regarding the structure and properties of pseudo-orbits in these spaces, and in Section Four we prove shadowing for dendritic Julia sets (Theorem \ref{thm:canonical_shadowing}), from which we obtain the proof of the main theorem, which states that for dendritic Julia sets, every ICT set is necessarily an $\omega$-limit set (Theorem \ref{thm:can_shad_ICT_property}).


\section{Preliminaries}

We examine the dynamics of quadratic Julia sets that are dendrites via Baldwin's encoding using non-Hausdorff itinerary spaces.  In this section we give a brief description of the definitions and results from \cite{Baldwin-JFPTA} that relate to dendritic Julia sets.  More detail and results can be found also in \cite{Baldwin-TopApp}.

Given a finite set of symbols, $A$, for each $n\in \N$ we denote the words of length $n$ from $A$ by $A^n$, and we define the set of finite words from $A$ by $$A^{<\w}=\bigcup_{n\in \N} A^n.$$  For a finite word $\al\in A^{<\w}$ let $len(\al)\in \N$ be the length of $\al$.

It is a standard practice in dynamics to encode complicated behavior via symbolic dynamics.  A typical simple encoding is to assign symbols from an alphabet to disjoint regions and then to track orbits by recording the regions the points traverse.  A difficulty arises due to the fact that the sequence space with its natural topology is totally disconnected while the dynamical system under consideration is usually not.  Therefore the encoding usually induces a semiconjugacy rather than a conjugacy.  Baldwin avoids this difficulty by the use of ``wildcard'' symbols. His encoding of the symbol space includes a $*$ symbol which stands for all of the other symbols at once.  The topology then on his set of symbols is not the discrete topology, but rather a slightly more complicated non-Hausdorff topology.

Consider the set $\{0,1,*\}$ with the non-Hausdorff topology with basis $\{\{0\}, \{1\}, \{0,1,*\}\}$.  Let $\Lambda$ be the product space induced by this topology on $\{0,1,*\}^\w$ ($\w$ is the set of non-negative integers, $\{0\}\cup \N$.)  For each $\al\in \{0,1,*\}^{<\w}$, we can define the basic open cylinder sets as
$$
B^\Lambda_\al\ :=\ \bigl\{\beta\in\Lambda\ :\ \beta_i=\alpha_i\mbox{ whenever }\alpha_i\neq*\mbox{, for }i\leq len(\alpha)\bigr\}.
$$

Notice that with these definitions, in the factor spaces $*$ cannot be separated via open sets from $1$ and $0$, so in the product space we cannot separate different infinite words if they only differ in places where one of them has a $*$.  This new topology on $\Lambda$ gives it many connected shift-invariant subspaces.  In fact Baldwin showed that it contains copies of every dendritic Julia set for a quadratic map.  We make this more precise below.

Let $\sh$ denote the natural shift map on a product space.  A sequence $\tau\in \Lambda$ is called \emph{$\Lambda$-acceptable} if, and only if, the following hold. \begin{enumerate}
  \item For all $n\in \w$, $\tau_n=*$ if, and only if $\sh^{n+1}(\tau)=\tau$.
  \item For all $n\in \w$ such that $\sh^n(\tau)\neq \tau$ there is an $m\in \w$ such that $*\neq \tau_{m+n}\neq \tau_m\neq *$, that is to say if $\sh^n(\tau)\neq \tau$ then these two sequences differ in a position where neither is $*$.
\end{enumerate}  The sequences $\tau$ which are $\Lambda$-acceptable are the sequences which we can view as \emph{kneading sequences}.  If $\tau$ is $\Lambda$-acceptable, then $\al\in \Lambda$ will be called \emph{$(\Lambda, \tau)$-consistent} if, and only if, for all $n\in \w$, $\al_n=*$ implies that $\sh^{n+1}(\al)=\tau$.  A sequence $\al\in \Lambda$ is called $(\Lambda,\tau)$-\emph{admissible} if, and only if, it is $(\Lambda, \tau)$-consistent and for all $n\in \w$ such that $\sh^n(\al)\neq *\tau$ then there is a position where these sequences differ and neither is a star (i.e. there is some $m>0$ such that $*\neq \al_{m+n}\neq \tau_{m-1}\neq *$).


\begin{defn}\label{def:D_tau}
Let $\C D_\tau\ :=\ \bigl\{x\in\Lambda\ :\ x\mbox{ is $(\Lambda, \tau)$-admissible}\bigr\}.$
\end{defn}

Thus $\C D_\tau$ is the set of all possible itineraries allowed by the kneading sequence $\tau$, about which Baldwin proved the following:

\begin{thm}\cite[Theorem 2.4]{Baldwin-JFPTA}
Let $\tau$ be $\Lambda$-acceptable.  Then $\C D_\tau$ is a shift-invariant self-similar dendrite.
\end{thm}

Baldwin proved that $\C D_\tau$ has a natural arc-length (or taxicab) metric, $d$, which we use in this paper. Moreover, this family of spaces includes all of the quadratic Julia sets which are dendrites:

\begin{thm}\cite[Theorem 2.5]{Baldwin-JFPTA}\label{thm:dendrite_conj}
Let $f_c(z)=z^2+c$.  If the Julia set of $f_c(z)$, $J_c$, is a dendrite then there is a $\Lambda$-acceptable sequence $\tau$ such that $f_c|_{J_c}$ is conjugate to $\sh|_{\C D_\tau}$.
\end{thm}


Let $\tau$ be a $\Lambda$-acceptable sequence.  We call $\tau$ the \emph{kneading sequence} for the self-similar dendrite $\C D_{\tau}$.  We call the point in $\C D_{\tau}$ of the form $*\tau$ the \emph{critical point} and all of its pre-images \emph{precritical points.}


Let $x=x_0x_1\ldots$ and $y=y_0y_1\ldots$ be non-precritical points, with $z=z_0z_1\ldots$ another (possibly precritical) point, and let $n\in \omega$. We define
\[
\restr{x}{n} := x_0x_1\ldots x_n.
\]
For every $i\leq n$, if $x_i = z_i$ whenever $z_i\neq*$, we say that $\restr{x}{n}$ and $\restr{z}{n}$ are \emph{equivalent}, and we write
$$
\restr{x}{n}\approx \restr{z}{n}.
$$
If instead $x_i=z_i$ for all $i\in\omega$ whenever $z_i\neq *$ then we write
$$
x\approx z.
$$
We say that
$$
\restr{x}{n}\simeq \restr{y}{n}
$$
if $\restr{x}{n} = \restr{y}{n}$ or there is a pre-critical point $z'$ for which
\[
\restr{x}{n}\approx \restr{z'}{n}\approx \restr{y}{n}.
\]
Given $\alpha\in\{0,1,*\}^{<\omega}$, the basic open cylinder sets for ${\C D_{\tau}}$ are
\[
B_{\alpha}^{\tau}\ :=\ \bigl\{x\in{\C D_{\tau}}\ :\ x_i=\alpha_i\mbox{ whenever }\alpha_i\neq*\mbox{, for }i\leq len(\alpha)\bigr\}.
\]

\medskip
Lemmas \ref{lem:precritical_equiv} and \ref{brentlemma} and Remark \ref{lem:kneading_equiv} are easy observations.

\begin{lem}\label{lem:precritical_equiv}
Suppose that $z_1$ and $z_2$ are precritical points, and $x$ is not precritical with
\[\restr{z_1}{n}\approx\restr{x}{n}\approx\restr{z_2}{n}.\]
for $n\in \omega$. Then $\restr{z_1}{n}\approx\restr{z_2}{n}$.
\end{lem}

\begin{defn}\label{def:returntime}
If $\tau$ is periodic, let $P$ be the period of $\tau$.  If $\tau$ is not periodic then for each integer $m\in \N$ we define the \emph{return time for $m$}, $r_m$, as $$r_m=\min\left(\{k\in \N:\restr{\sh^k(\tau)}{m}=\restr{\tau}{m}\}\cup\{\infty\}\right).$$
\end{defn}

It is easy to see that if $\tau$ is non-recurrent then there is some $M$ such that for all $m\ge M$, $r_m=\infty$.  Also if $\tau$ is recurrent but not periodic then $r_m$ is an integer for each $m\in\nat$ and $r_m\to \infty$ as $m\to \infty$.

\begin{rem}\label{lem:kneading_equiv}
Notice that if $k,j\in\nat$, with $k>j$ such that
\[
\restr{\tau}{j} \approx \restr{\sigma^{k-j}(\tau)}{j},
\]
then
\begin{enumerate}
	\item if $\tau$ is periodic with period $P$ we either have that $k-j=nP$ for some $n\in\nat$, or $j<P$;
	\item if $\tau$ is not periodic we have that $k-j\geq r_{j}$.
\end{enumerate}
\end{rem}

The next result is central to our understanding of dendrite maps, and follows from the definition of the cylinder sets above.

\begin{lem}\label{brentlemma} For $x,y\in{\C D_{\tau}}$ and $\eps>0$, there is an $N_{\eps}\in \N$ such that
$$d(x,y)<\eps$$
if, and only if
$$x\rtr_{N_{\eps}}\simeq y\rtr_{N_{\eps}}.$$
\end{lem}

This lemma describes the key difference between the setting of this paper and the setting of shifts of finite type. In SFTs two points are within $\eps$ of each other if and only if they are \emph{equal} for their initial sement of length $N_{\eps}$.  In the symbolic spaces we study in this paper,  we know that two points $x$ and $y$ are within $\eps$ of each other if, and only if either:
\begin{enumerate} \item $x$ and $y$ have exactly the same initial segment of length $N_{\eps}$, or
\item $x$ and $y$ have a disagreement in say the $j$th place (with $j\le N_{\eps}$), in which case there is a precritical point $z$ between them that has a $*$ in the $j$th place.
\end{enumerate}

\begin{defn}\label{def:flip}
In case (2) we say that $x$ and $y$ have a \emph{flip} in the $j$th place.
\end{defn}

For example we could have two points $x$ and $y$ within $\eps$ of each other as below:

\begin{align*} x=x_0x_{1}x_{2}\dots x_{j-1}&x_{j}x_{j+1}\dots x_{N_{\eps}}\\
z=x_0x_{1}x_{2}\dots x_{j-1}&*\tau_{1}\tau_{2}\dots \tau_{N_{\eps}-j}\\
y=x_0x_{1}x_{2}\dots x_{j-1}&y_{j}y_{j+1}\dots y_{N_{\eps}}\\
\end{align*}
where $j$ is chosen minimally. In this case we must have that the words $x_{j+1} \dots x_{N_{\eps}}$ and $y_{j+1}\dots y_{N_{\eps}}$ are equivalent to $\tau_{1}\dots \tau_{N_{\eps}-j}$. So if $\tau$ is not periodic then it must be the case that these three words are equal. If instead $\tau$ is periodic, say with period $P$, then these words must be equal for the first $P$ many symbols, then they can have another flip, then they must be equal for the next $P$ many symbols, etc.  Analyzing this situation for $\delta$ pseudo-orbits and using it to construct $\eps$-shadowing points is the focus of the next section.


\section{Pseudo-Orbits in Dendrites}

In this section we will prove a number of results pertaining to pseudo-orbits in dendrites, which will ultimately allow us to prove that dendrite maps have shadowing.

Let $\delta>0$.  For the purpose of illustration, we begin with a consideration of $\delta$ pseudo-orbits in SFT spaces. Let$\{x_{i}\}_{i\in \N}$ be a $\delta$ pseudo-orbit in $X$, a SFT. Let $N_{\delta}	\in \N$ be defined so that $$d(x,y)<\delta$$ if, and only if $$x\rtr_{N_{\delta}}=y\rtr_{N_{\delta}}.$$ By definition we have $$\sh(x_{1})\rtr_{N_{\delta}}= x_{2}\rtr_{N_{\delta}}$$ and $$\sh(x_{2})\rtr_{N_{\delta}}=x_{3}\rtr_{N_{\delta}}$$ which implies $$\sh^{2}(x_{1})\rtr_{n_{\delta}-1}=\sh(x_{2})\rtr_{N_{\delta}-1}=x_{3}\rtr_{N_{\delta}-1}.$$ If we denote $x_{i}$ by $x^{i}_{0}x^{i}_{1}x^{i}_{2}\dots$ then we have the following array of initial segments of length $N_\delta$ where down each column we have equality:

\begin{align*} x^{1}_{0}x^{1}_{1}x^{1}_{2}x^{1}_{3}x^{1}_{4}&\dots x^{1}_{N_{\delta}}\\
x^{2}_{0}x^{2}_{1}x^{2}_{2}x^{2}_{3}&x^{2}_{4}\dots x^{2}_{N_{\delta}}\\
x^{3}_{0}x^{3}_{1}x^{3}_{2}&x^{3}_{3}x^{3}_{4}\dots x^{3}_{N_{\delta}}\\
x^{4}_{0}x^{4}_{1}&x^{4}_{2}x^{4}_{3}x^{4}_{4}\dots x^{4}_{N_{\delta}}\\
x^{5}_{0}&x^{5}_{1}x^{5}_{2}x^{5}_{3}x^{5}_{4}\dots x^{5}_{N_{\delta}}\\
\end{align*}

The column equality in the above array allows a straightforward proof of shadowing in SFT spaces. For sufficiently small $\delta$ the point $z=x^{1}_{0}x^{2}_{0}x^{3}_{0}\dots$ is in the space, and it also $\delta$ shadows the pseudo-orbit. It is easy to see from the array that $$z\rtr_{N_{\delta}}=x_{1}\rtr_{N_{\delta}}.$$ This is also true for the $j$th shift of $z$ and the point $x_{j}$.

The symbolic spaces we consider in this paper are far less simple.  Let $\{x_{i}\}_{i\in\N}$ be a $\delta$ pseudo-orbit in $\C D_{\tau}$, and let $N_{\delta}\in \N$ be given as in Lemma \ref{brentlemma}. Then, as described after the lemma, for each $i\in\nat$ we have $$\sh(x_{i})\rtr_{N_{\delta}}\simeq x_{i+1}\rtr_{N_{\delta}}.$$  This implies that either these two initial segments are equal or there is a flip in some position, and a precritical point, $z$, between $\sh(x_{i})$ and $x_{i+1}$.

Consider now the array associated with the $\delta$ pseudo-orbit:

\begin{align*} x^{1}_{0}x^{1}_{1}x^{1}_{2}x^{1}_{3}x^{1}_{4}&\dots x^{1}_{N_{\delta}}\\
x^{2}_{0}x^{2}_{1}x^{2}_{2}x^{2}_{3}&x^{2}_{4}\dots x^{2}_{N_{\delta}}\\
x^{3}_{0}x^{3}_{1}x^{3}_{2}&x^{3}_{3}x^{3}_{4}\dots x^{3}_{N_{\delta}}\\
x^{4}_{0}x^{4}_{1}&x^{4}_{2}x^{4}_{3}x^{4}_{4}\dots x^{4}_{N_{\delta}}\\
x^{5}_{0}&x^{5}_{1}x^{5}_{2}x^{5}_{3}x^{5}_{4}\dots x^{5}_{N_{\delta}}\\
&\ddots \, \ddots \, \ddots \, \, \, \,\,\, \, \, \, \ddots 	 \\
\end{align*}
Since this is a $\delta$ pseudo-orbit we know that
\begin{equation}\label{denequal1}
\sh^{t}(x_{1})\rtr_{N_{\delta}-t} \app \sh^{t-1}(x_{2})\rtr_{N_{\delta}-t} \app \dots \app x_{t+1}\rtr_{N_{\delta-t}}
\end{equation}
for every $0\leq t\leq N_{\delta}$. Notice also that
\begin{equation}\label{sftequal}
\sh^{t}(x_{1})\rtr_{N_{\delta}-t}=\sh^{t-1}(x_{2})\rtr_{N_{\delta}-t}=\dots =x_{t+1}\rtr_{N_{\delta-t}}
\end{equation}
for all $t \le N_{\delta}$ if, and only if, we have column equality in these segments as in the SFT case. But by Lemma \ref{brentlemma} it may be the case that there is a flip in some column, say the $j$th column, relative to $x_1$. If $j$ is minimal in this respect, then for all $1 \leq \ell < j$ we have the following column equality in the array:
$$
x^{1}_{\ell} = x^{2}_{\ell-1} = \dots = x^{\ell+1}_{0}.
$$
This does not hold in the $j$th column, since for some $1\leq i \leq j$ we have
$$
x^{1}_{j} = x^{2}_{j-1} = \dots = x^{i}_{j-i+1}\neq x^{i+1}_{j-i}.
$$
This motivates the following definition:


\begin{defn}\label{def:flip_row_column}
We will refer to any such $j$ above as a \emph{flip column relative to $x_{1}$} and $i$ as a \emph{flip row relative to $x_{1}$}. Given $x_t$ for $t\geq1$ we define flip columns relative to $x_t$ in an analogous manner, and given a flip column $j$ relative to $x_t$, we define the \emph{flip row of the flip column $j$} to be the least $i$ such that
$$x^t_j\neq x^{t+i}_{j-i}.$$
\end{defn}

These flip rows and flip columns are the main obstruction to proving shadowing and characterizing $\w$-limit sets for the dendrite maps under consideration.

In SFTs, letting $\delta=\eps$, one can prove every $\delta$ pseudo-orbit is $\eps$ shadowed by a point like $z$ defined above \cite{Walters}.  We would like to mirror the proof from SFT spaces in dendrites, but we must make some significant changes to account for the flip columns.  Given $\eps>0$ we will choose a $\delta>0$ (based on several upcoming lemmas) much smaller than $\eps$.  Then, given a $\delta$ pseudo-orbit $\{x_i\}_{i\in \N}$, we will construct a ``canonical form'' of an $\eps$-shadowing point, $z=z_0z_1\dots$ where $z_0=x^1_0$ and if $t_1$ is the first flip column for the pseudo-orbit we let $z_j=x^{j+1}_0$ for all $j<t_1$ but we define $z_{t_1}=\diamond$.  Letting $t_2$ be the next flip column we define $z_k=x^{k+1}_0$ for all $t_1<k<t_2$ and again we assign $\diamond$ to $z_{t_2}$.  Continuing this procedure gives us a point $\hat{z}$ with possibly infinitely many places where there is a $\diamond$, but since $\diamond$ is not in our alphabet, $\hat{z}$ is not a point in $\C D_\tau$.  We must then assign either a $0$, $1$ or $*$ to each $\diamond$ in $\hat{z}$ in such a way that the resulting point $z$ is both in $\C D_\tau$ and $\eps$-shadows the pseudo-orbit.

It is the goal of the next several results to prove that there is a $\delta$ small enough so that the $\diamond$'s will occur with larger than $N_\eps$ gaps between them.  Once this has been established, we then prove that \textit{any} assignment of $0$'s or $1$'s to the $\diamond$'s in $\hat{z}$ will generate a true $\eps$-shadow $z$ of the $\delta$ pseudo-orbit (unless $\sigma^k(\hat{z})\app\tau$ for some $k\in\w$, in which case there is a unique assignment). 


To begin, we prove a basic feature of column equality in arrays related to $\delta$ pseudo-orbits between successive flip columns.  The proof is immediate from the definition of flip column (see Figures \ref{fig:flip_equality1} and \ref{fig:flip_equality2}).

\begin{lem}\label{equalitybetweenflips} Let $\{x_i\}_{i\in \N}$ be a $\delta$ pseudo-orbit, and let $k\in \N$ be chosen such that there are two flip columns, $j_1$ and $j_2$ relative to $x_k$ with $j_1<j_2$ and no flip columns relative to $x_k$ between $j_1$ and $j_2$.  Then we have the following equalities:\begin{enumerate}
  \item for all $1\le \ell\le k+j_1$, $$\sh^{j_1+1}(x_k)\rtr_{j_2-j_1-1}
      =\sh^{j_1+1-\ell}(x_{k+\ell})\rtr_{j_2-j_1-1},$$ and
  \item for all $1\le n\le j_2-j_1-1$ $$\sh^{j_1+n}(x_k)\rtr_{j_2-j_1-n}=x_{k+j_1+n}\rtr_{j_2-j_1-n}.$$
\end{enumerate}
\end{lem}


\begin{figure}[ht]
\begin{center}
\psset{xunit=1.8cm,yunit=1.8cm,algebraic=true,dotstyle=o,dotsize=3pt 0,linewidth=0.8pt,arrowsize=3pt 2,arrowinset=0.25}
\begin{pspicture*}(0.5,0.5)(6.5,5)
\psline(1.0,4)(6,4)
\psline(1.75,3)(6,3)
\psline(2.8,2)(6,2)
\psline(4.7,1)(6,1)
\rput[tl](0.75,4.1){\textbf{$x_k$}}
\rput[tl](1.25,3.1){\textbf{$x_{k+l}$}}
\rput[tl](2.25,2.1){\textbf{$x_{k+j_1}$}}
\rput[tl](4.2,1.1){\textbf{$x_{k+j_2}$}}
\psline[linestyle=dotted](2.8,4.0)(2.8,2.0)
\psline[linestyle=dotted](4.7,4.0)(4.7,1.0)

\psline{->}(2.95,4.4)(1.0,4.4)
\psline{->}(1.0,4.4)(2.95,4.4)

\psline{->}(2.95,4.4)(4.55,4.4)
\psline{->}(4.55,4.4)(2.95,4.4)

\psline(2.95,3.975)(4.55,3.975)
\psline(2.95,4.025)(4.55,4.025)
\psline(2.95,2.975)(4.55,2.975)
\psline(2.95,3.025)(4.55,3.025)

\rput[tl](3.25,3.6){$Agreement$}
\psline{->}(3.8,3.6)(3.8,3.9)
\psline{->}(3.8,3.4)(3.8,3.1)
\rput[tl](1.75,4.55){\textbf{$^{j_1+1}$}}
\rput[tl](2.6,4.2){\textbf{$^{j_1}$}}
\rput[tl](4.65,4.2){\textbf{$^{j_2}$}}
\rput[tl](3.55,4.55){\textbf{$^{j_2-j_1-1}$}}
\rput[tl](2.75,2.65){$\star$}
\psline(2.8,2.45)(2.8,2.55)
\rput[tl](4.65,1.65){$\star$}
\psline(4.7,1.45)(4.7,1.55)
\psline[linestyle=dotted](2.3,2.5)(6.0,2.5)
\psline[linestyle=dotted](3.8,1.5)(6.0,1.5)
\end{pspicture*}
\caption{Lemma \ref{equalitybetweenflips} (1).}\label{fig:flip_equality1}
\end{center}
\end{figure}


\begin{figure}[ht]
\begin{center}
\psset{xunit=1.8cm,yunit=1.8cm,algebraic=true,dotstyle=o,dotsize=3pt 0,linewidth=0.8pt,arrowsize=3pt 2,arrowinset=0.25}
\begin{pspicture*}(0.5,0.5)(6.5,5)
\psline(1.0,4)(6,4)
\psline(2.8,3)(6,3)
\psline(3.4,2)(6,2)
\psline(4.7,1)(6,1)
\rput[tl](0.75,4.1){\textbf{$x_k$}}
\rput[tl](2.25,3.1){\textbf{$x_{k+j_1}$}}
\rput[tl](2.6,2.1){\textbf{$x_{k+j_1+n}$}}
\rput[tl](4.2,1.1){\textbf{$x_{k+j_2}$}}
\psline[linestyle=dotted](2.8,4.0)(2.8,3.0)
\psline[linestyle=dotted](4.7,4.0)(4.7,1.0)

\psline{->}(3.4,4.4)(1.0,4.4)
\psline{->}(1.0,4.4)(3.4,4.4)

\psline{->}(3.4,4.4)(4.55,4.4)
\psline{->}(4.55,4.4)(3.4,4.4)

\psline(3.4,3.975)(4.55,3.975)
\psline(3.4,4.025)(4.55,4.025)
\psline(3.4,1.975)(4.55,1.975)
\psline(3.4,2.025)(4.55,2.025)

\rput[tl](3.5,2.88){$Agreement$}
\psline{->}(4.0,2.9)(4.0,3.9)
\psline{->}(4.0,2.7)(4.0,2.1)
\rput[tl](2.0,4.55){\textbf{$^{j_1+n}$}}
\rput[tl](2.7,4.2){\textbf{$^{j_1}$}}
\rput[tl](4.65,4.2){\textbf{$^{j_2}$}}
\rput[tl](3.6,4.55){\textbf{$^{j_2-j_1-n}$}}
\rput[tl](2.75,3.65){$\star$}
\psline(2.8,3.45)(2.8,3.55)
\rput[tl](4.65,1.65){$\star$}
\psline(4.7,1.45)(4.7,1.55)
\psline[linestyle=dotted](1.25,3.5)(6.0,3.5)
\psline[linestyle=dotted](3.8,1.5)(6.0,1.5)
\end{pspicture*}
\caption{Lemma \ref{equalitybetweenflips} (2).}\label{fig:flip_equality2}
\end{center}
\end{figure}


Before proceeding with a proof that these dendrite maps have shadowing, we need to carefully consider both periodic and non-periodic kneading sequences and the gap between successive flip columns.  First we consider the periodic case.  Let $\tau$ be periodic, $\delta>0$ and let $N_\delta$ be defined as in Lemma \ref{brentlemma}.  Let $\{x_i\}_{i\in \N}$ be a $\delta$ pseudo-orbit.  Since $\tau$ is periodic, if $x_k$ has a flip column in say the $j$th column with $j<N_\delta$, then it can have a flip column in every $j+kP$ column afterwards, but also in some columns which are not ``in sync'' with the period.  The next lemma addresses the arrangement of the flip columns that are ``out of sync''.

\begin{lem}\label{periodicsublemma}
Let $\{x_i\}_{i\in \N}$ be a $\delta$ pseudo-orbit, and let $k$ and $j<N_\delta$ be chosen such there there is a pre-critical point $z$ with
$$
\sh(x_{k})\rtr_{N_\delta}\app z\rtr_{N_\delta}\app x_{k+1}\rtr_{N_\delta}
$$
and $j$ is minimal with $\sh^j(z)=*\tau_1\tau_2\dots \tau_{P-1}*\tau_1\dots$. Suppose also that there are flip columns $t$ for $x_k$ such that $t-j$ is not a multiple of $P$.

Choose $s_1$ minimal such that $s_1$ is a flip row of a flip column $t_1$ with $t_1-j$ not a multiple of $P$ and
$$
t_1<N_\delta.
$$
Given $s_i$ and $t_i$, define $s_{i+1}>s_i$ to be minimal such that $s_{i+1}$ is a flip row of a flip column $t_{i+1}$ with $t_{i+1}-j$ not a multiple of $P$ and with
$$
t_{i+1}<t_i.
$$
Then there are finitely many such $s_i$'s and $t_i$'s, and for each such $i$ we have
$$
N_{\delta}-t_i<iP.
$$
 \end{lem}

\begin{proof}
  We start with the flip row $s_1$ and flip column $t_1$.  Since $s_1$ is the least flip row associated with $x_k$ with $t_1-j$ not a multiple of $P$, there are no flips in the array in non-$j+kP$ places between $\sh^{s_1}(x_k)$ and $x_{k+s_1}$ until the $t_1$ symbol.  There must be a precritical point $z_1$ with
$$
\sh(x_{k+s_1-1})\rtr_{N_\delta}\app z_1\rtr_{N_\delta} \app x_{k+s_1}\rtr_{N_\delta}
$$
with
$$
\sh^{t_1-s_1}(z_1)=\tau.
$$
But since $t_1$ is a flip column for $x_k$ and $s_1$ is the least flip row, it must be the case that
$$
\sh^{t_1+1}(z)=\sh^{t_1+1-j}(\tau)\neq \tau_1\tau_2\dots
$$
but
$$
\sh^{t_1+1-j}(\tau)\rtr_{N_\delta-t_1}\app \sh^{t_1+1-s_1}(x_{k+s_1})\rtr_{N_\delta-t_1}
$$
and
$$
\sh^{t_1+1-s_1}(x_{k+s_1})\rtr_{N_\delta-t_1}\app \sh^{t_1+1-s_1}(z_1)\rtr_{N_\delta-t_1} =  \tau_1\tau_2\dots\tau_{N_\delta-t_1}.
$$
Therefore
$$
\sh^{t_1+1-j}(\tau)\rtr_{N_\delta-t_1}\app \tau_1\tau_2\dots\tau_{N_\delta-t_1}.
$$
By Remark \ref{lem:kneading_equiv}(1) it must be the case that $N_\delta-t_1<P$.

\begin{figure}[ht]
\begin{center}
\psset{xunit=2.0cm,yunit=2.0cm,algebraic=true,dotstyle=o,dotsize=3pt 0,linewidth=0.8pt,arrowsize=3pt 2,arrowinset=0.25}
\begin{pspicture*}(0.5,1.75)(6.5,5)
\psline(1.0,4)(6,4)
\psline(1.85,3)(6,3)
\psline(2.5,2)(6,2)

\rput[tl](0.7,4.1){\textbf{$x_k$}}
\rput[tl](1.3,3.1){\textbf{$x_{k+s_1}$}}
\rput[tl](1.8,2.1){\textbf{$x_{k+s_2}$}}

\psline[linestyle=dotted](5.2,4.0)(5.2,2.0)
\psline[linestyle=dotted](5.7,4.0)(5.7,2.0)
\psline[linestyle=dotted](2.85,4)(2.85,2.0)

\psline{->}(1.0,4.42)(6.0,4.42)
\psline{->}(6.0,4.42)(1.0,4.42)

\psline{->}(6.0,2.5)(4.8,2.5)
\psline{->}(4.8,2.5)(6.0,2.5)

\psline{->}(5.4,3.5)(6.0,3.5)
\psline{->}(6.0,3.5)(5.4,3.5)

\rput[tl](3.38,4.7){\textbf{$N_{\delta}$}}
\rput[tl](5.35,2.65){\textbf{$^{2P}$}}
\rput[tl](5.5,3.65){\textbf{$^{P}$}}
\rput[tl](2.8,4.2){\textbf{$^{j}$}}
\rput[tl](5.12,4.2){\textbf{$^{t_2}$}}
\rput[tl](5.62,4.2){\textbf{$^{t_1}$}}
\rput[tl](2.8,3.925){$\star$}
\psline(2.85,3.95)(2.85,4.05)
\rput[tl](5.15,2.15){$\star$}
\psline(5.2,1.95)(5.2,2.05)
\rput[tl](5.65,3.15){$\star$}
\psline(5.7,2.95)(5.7,3.05)
\end{pspicture*}
\caption{Lemma \ref{periodicsublemma}.}
\end{center}
\end{figure}

The result follows by induction noticing that $\sh^{t_i-s_i}(x_{k+s_i})$ will agree with $\sh^{t_i}(x_k)$ on a word of length $t_{i-1}-t_i$ which by similar reasoning must be less than $P$.  We combine this with $N_\delta-t_{i-1}<(i-1)P$ to see that $$N_\delta-t_i=(N_{\delta}-t_{i-1})+(t_{i-1}-t_i)<(i-1)P+P=iP.$$
\end{proof}

Next we consider the case that $\tau=*\tau_1\tau_2\dots$ is not periodic. The following combinatorial argument will be a useful tool when dealing with return times. For words $\alpha$ and $\beta$, when we write $\alpha\beta$ we mean $\alpha$ concatenated with $\beta$; for $n\in\nat$, when we write $\alpha^n$ we mean $\alpha$ concatenated with itself $n$ times.

\begin{lem}\label{lem:word_repetition}
Let $\alpha$ be a word of length $n$, let $\beta=\alpha\alpha$, and suppose that for some $1\leq m<n$ we have that $\sigma^{n-m}(\beta)$ begins with the word $\beta$. Then there is some word $\gamma$ of length $\ell$ with $1\leq \ell\leq m$, such that $\ell$ divides $m$ and $n-m$ and $\alpha=\gamma^{n/\ell}$.
\end{lem}
\begin{proof}
Set $m_1=m$ and $\alpha_0=\alpha$. From the statement of the lemma, we get that $\sigma^{n-m_1}(\beta)$ begins with the initial $m_1$-segment of $\alpha_0$; call this segment $\alpha_1$. Then  $\sigma^n(\beta)$, which is also the word $\alpha_0$, begins with $\alpha_1$ and this also corresponds to the second $m_1$-segment of the occurrence of $\alpha_0$ beginning at $\sigma^{n-m_1}(\beta)$, provided $n\geq2m_1$ (see Figure \ref{fig:overlap}).

\begin{figure}[ht]
	\begin{center}
\begin{pspicture*}(-0.5,2.5)(11.5,4.5)
\psset{xunit=1.3cm,yunit=1.0cm,algebraic=true,dotstyle=o,dotsize=3pt 0,linewidth=1.0pt,arrowsize=3pt 2,arrowinset=0.25}
\psaxes[labelFontSize=\scriptstyle,xAxis=false,yAxis=false,Dx=1,Dy=1,ticksize=-2pt 0,subticks=2]{->}(0,0)(-0.5,2.5)(8.5,4.5)
\psline(0,3)(8,3)
\psline(3,4)(7,4)
\psline(4,2.7)(4,3.3)
\psline(0,2.7)(0,3.3)
\rput[tl](1.0,3.3){$\mathbf{\alpha_0}$}
\rput[tl](5.8,3.3){$\mathbf{\alpha_0}$}
\rput[tl](4.94,4.3){$\mathbf{\alpha_0}$}
\psline[linewidth=0.5pt,linecolor=gray](4,3.7)(4,4.3)
\rput[tl](3.3,4.5){$(\alpha_1)$}
\psline[linewidth=0.5pt,linecolor=gray](5,2.7)(5,3.3)
\rput[tl](4.3,3.5){$(\alpha_1)$}
\psline(8,2.7)(8,3.3)
\psline(3,4.3)(3,3.7)
\psline(7,4.3)(7,3.7)
\rput[tl](1.95,2.88){$n$}
\psline[linewidth=0.5pt]{->}(2.15,2.8)(4,2.8)
\psline[linewidth=0.5pt]{->}(1.96,2.8)(0,2.8)
\rput[tl](3.35,3.81){$m_1$}
\psline[linewidth=0.5pt]{->}(3.75,3.75)(4,3.75)
\psline[linewidth=0.5pt]{->}(3.3,3.75)(3,3.75)
\rput[tl](1.44,3.9){$n-m_1$}
\psline[linewidth=0.5pt]{->}(2.4,3.75)(3.0,3.75)
\psline[linewidth=0.5pt]{->}(1.37,3.75)(0,3.75)
\end{pspicture*}
	\end{center}
	\caption{\small{Overlapping words.}}\label{fig:overlap}
\end{figure}
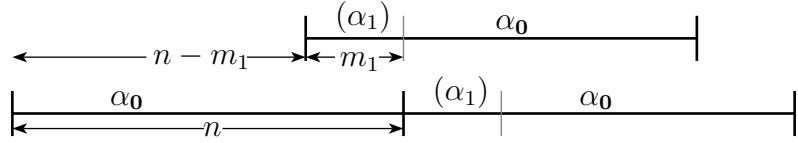

Continuing in this way we see that either
\begin{enumerate}
	\item $m_1$ divides $n$, so $\alpha_0=\alpha_1^{n/m_1}$, or
	\item $\alpha_0 = (\alpha_1^{k_1})\alpha_2$ for some $k_1\geq1$ and some word $\alpha_2$ of length $m_2<m_1$.
\end{enumerate}

In case (1) we are done. In case (2), if we let $\beta_1=\alpha_1\alpha_1$, from the structure of the word $B$ we get that $\sigma^{m_1-m_2}(\beta_1)$ begins with $\alpha_1$. But this is precisely the situation we started with, so we can iterate the above argument. Since every finite word is made up of component symbols, this process cannot continue indefinitely, so there must be some word $\alpha_r$ of length $m_r\geq1$ such that $\alpha_0=\alpha_r^{n/m_r}$.

To see that $m_r$ divides $m_1$, notice that $n=k_1m_1+m_2$, where $m_2<m_1$, and similarly $m_1=k_2m_2+m_3$ for $k_2,m_3\in\nat$ with $m_3<m_2$. This is the Euclidean Algorithm, which repeats with $m_i$ being the length of $\alpha_i$, until we get that $m_{r-1}=k_rm_r$, and $m_r$ is the greatest common divisor of $n$ and $m_1$.
\end{proof}

\begin{lem}\label{lem:return_times}
Suppose that $\tau$ is not periodic. Then the return times $\{r_m\}_{m\in \N}$ have the property that $r_t\geq t$ for infinitely many $t\in\nat$.
\end{lem}
\begin{proof}
Recall first that $r_m\leq r_{m+1}$ for every $m\in\nat$. Suppose that the statement of the lemma fails; in other words we have that $r_m<m$ for cofinitely many $m\in\nat$. Then we can choose $m'$ maximum such that $r_{m'}\geq m'$, and we get that $r_{m'+1}<m'+1$. Thus
\begin{align*}
m' &\leq r_{m'}\leq r_{m'+1}\leq m',\\
\mbox{so }r_{m'} &= r_{m'+1} = m'.
\end{align*}
This gives us that the first $m'$-segment of $\tau$ (which follows the initial symbol $*$ of $\tau$) is repeated immediately; let us refer to this segment as $\alpha$.

Let $i\geq2$ be minimal such that $r_{m'+i}\neq m'$. Then
\begin{align*}
m' &= r_{m'+i-1} < r_{m'+i} < m'+i,\\
\mbox{and so }r_{m'+i} &\leq m'+i-1.
\end{align*}

Let $r_{m'+i}=m'+j$ for $1\leq j < i$, and pick $k\in\nat$ such that $(k-1)m' < m'+j \leq km'$ (see Figure \ref{fig:returntime}).

\begin{figure}[ht]
	\begin{center}
\begin{pspicture*}(-2.0,2.4)(10.0,4.8)
\psset{xunit=1.0cm,yunit=1.0cm,algebraic=true,dotstyle=o,dotsize=3pt 0,linewidth=1.0pt,arrowsize=3pt 2,arrowinset=0.25}
\psaxes[labelFontSize=\scriptstyle,xAxis=false,yAxis=false,Dx=1,Dy=1,ticksize=-2pt 0,subticks=2]{->}(0,0)(-0.5,2.5)(8.5,4.5)
\psline(0,3)(4,3)
\psline(5,3)(9,3)
\psline(6,4)(9,4)
\rput[tl](9.1,4.03){$\ldots$}
\rput[tl](9.1,3.03){$\ldots$}
\psline(2,2.7)(2,3.3)
\psline(4,2.7)(4,3.3)
\psline(0,2.7)(0,3.3)
\psline(9,2.7)(9,3.3)
\psline(5,2.7)(5,3.3)
\psline(7,2.7)(7,3.3)
\psline(6,4.3)(6,3.7)
\psline(8,4.3)(8,3.7)
\rput[tl](5.5,4.75){$^{r_{m'+i}=m'+j}$}
\rput[tl](-1.3,3.1){$\tau=$}
\rput[tl](-0.55,3.2){$(\star)$}
\rput[tl](1,3.3){$\mathbf{\alpha}$}
\rput[tl](3,3.3){$\mathbf{\alpha}$}
\rput[tl](6,3.3){$\mathbf{\alpha}$}
\rput[tl](8,3.3){$\mathbf{\alpha}$}
\rput[tl](7,3.86){$\mathbf{\alpha}$}
\rput[tl](4.25,3.0){$\ldots$}
\rput[tl](1.75,3.7){$^{r_{m'}=m'}$}
\rput[tl](3.8,3.7){$^{2m'}$}
\rput[tl](4.75,2.7){$^{(k-1)m'}$}
\rput[tl](6.8,2.7){$^{km'}$}
\end{pspicture*}
	\end{center}
	\caption{\small{Return times overlapping.}}\label{fig:returntime}
\end{figure}
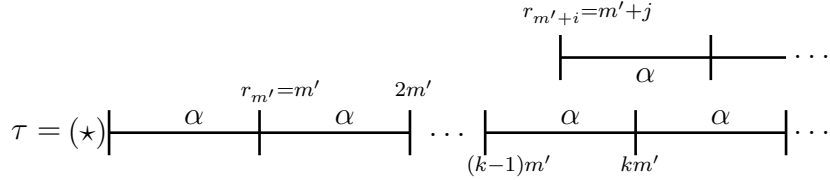

We claim that for any such $i$, the $(m'+i)$-segment of $\tau$ following the first $(r_{m'+i})$-segment will be an exact repeat of the first $(r_{m'+i})$-segment. This will force the $m'$-segment after the initial symbol of $\tau$ to repeat indefinitely, which is not an acceptable sequence for $\tau$, and the resulting contradiction would complete the proof. We prove that this occurs for the current $i$ under consideration, and a simple argument will conclude that it occurs for all $r_m>r_{m'+i}$.

If $m'+j=km'$ we can stop immediately. Otherwise $(k-1)m' < m'+j < km'$, then Lemma \ref{lem:word_repetition} tells us that there is some $\ell\leq m'$ such that each identical $m'$-segment of $\tau$ is composed of $\ell$ identical words of length $m'/\ell$. In this case we have that $\ell$ divides $(k-1)m'-j$ and we have the required repetition.

To conclude the proof, notice that for any $k\geq1$,
\[
r_{m'+i+k-1} \leq r_{m'+i+k} \leq m'+i+k-1,
\]
so when the initial $(m'+i+k)$-segment (following the initial symbol) first repeats in $\tau$ it does so before the end of the repeat of the first $(m'+i+k-1)$-segment. This forces an overlap, and as above we get that the first repeat of the initial $(m'+i+k)$-segment must exactly repeat what has come before it.
\end{proof}

\begin{lem}\label{nonperiodicsublemma}  Let $\{x_i\}_{i\in\nat}$ be a $\delta$ pseudo-orbit, let $\tau$ be a recurrent non-periodic kneading sequence, and let $\{m_i\}_{i\in \N}$ be a sequence of natural numbers chosen such that $$m_i\le r_{m_i}<r_{m_i}+1<m_{i+1}.$$  Suppose that $k$ is chosen such that there are two successive flip columns $j_1<j_2$ relative to $x_k$ with flip rows $i_1$ and $i_2$.  If there exists some $t\in \N$ such that
\begin{enumerate}
  \item[(i)] $j_2-j_1-1<m_t$, and
  \item[(ii)] $N_\delta-j_2-1>m_{t+1}$,
\end{enumerate}
then there is a third flip column, $j_3>j_2$, relative to $x_k$ with flip row $i_3$ such that
\begin{enumerate}
  \item[(a)] $i_3$ is between $i_1$ and $i_2$, and
  \item[(b)] $j_3-j_2-1<m_{t+1}$.
\end{enumerate}
\end{lem}

\begin{proof}
  Suppose we have the points, flip columns and rows as described in (i) and (ii) of the lemma, but suppose that there is no third flip column, $j_3$ satisfying (a) and (b).

  Without loss of generality, assume that $i_1<i_2$.  Then there is some pre-critical point $z_1\in [x_{k+i_i-1}, x_{k+i_1}]$ such that $$\sh(x_{k+i_1-1})\rtr_{N_\delta}\app z_1\rtr_{N_\delta}\app x_{k+i_1}\rtr_{N_\delta},$$ and $\sh^{j_1-i_1}(z_1)=\tau=*\tau_1\tau_2\tau_3\dots$.  Hence $$\sh^{j_1-i_1+1}(x_{k+i_1})\rtr_{N_\delta-j_1}\app \tau_1\tau_2\tau_3\dots\rtr_{N_\delta-j_1}.$$

  Similarly, there is a pre-critical point $z_2\in [x_{k+i_2-1}, x_{k+i_2}]$ such that $$\sh(x_{k+i_2-1})\rtr_{N_\delta}\app z_2\rtr_{N_\delta}\app x_{k+i_2}\rtr_{N_\delta},$$ and $\sh^{j_2-i_2}(z_2)=\tau=*\tau_1\tau_2\tau_3\dots$.  Hence $$\sh^{j_2-i_2+1}(x_{k+i_2})\rtr_{N_\delta-j_2}\app \tau_1\tau_2\tau_3\dots\rtr_{N_\delta-j_2}.$$  We see that $$\tau_{j_2-j_1+1}\tau_{j_2-j_1+2}\dots =\sh^{j_2-j_1+1}(\tau)=\sh^{j_2-i_1+1}(z_1)$$ and $$\sh^{j_2-i_1+1}(z_1)\rtr_{N_\delta-j_2-1}\app \sh^{j_2-i_1+1}(x_{k+i_1})\rtr_{N_{\delta}-j_2-1}.$$ While in the same column we have $$\tau_1\tau_2\dots=\sh^{j_2-i_2+1}(z_2),$$ and $$\sh^{j_2-i_2+1}(z_2)\rtr_{N_\delta-j_2-1}\app \sh^{j_2-i_2+1}(x_{k+i_2})\rtr_{N_{\delta}-j_2-1}.$$  Since we assume that there are no successive flip columns $j_3$ with related flip row $i_3$ where $i_1<i_3<i_2$ and we assume $j_3-j_2-1<m_{t+1}$ and $N_\delta-j_2-1>m_{t+1}$, it must be the case that \begin{align}\label{sublemmaequation}\sh^{j_2-i_1+1}(z_1)\rtr_{M}\app \sh^{j_2-i_1+1}&(x_{k+i_1})\rtr_{M}\app\\ \notag &\sh^{j_2-i_2+1}(x_{k+i_2})\rtr_{M}\app \sh^{j_2-i_2+1}(z_2)\rtr_{M}\end{align} for some $M\in \N$ with $M\ge m_{t+1}$ (otherwise there \emph{would} have been a flip column $j_3$ with flip row $i_3$ satisfying (a) and (b)).  There are three cases to consider.  For instance we could have \begin{enumerate} \item[(A)] $z_1\neq x_{k+i_1}$ and $x_{k+i_1}$ is precritical, or \item[(B)] $z_1=x_{k+i_1}$, or \item[(C)] $z_1\neq x_{k+i_1}$ and $x_{k+i_1}$ is not precritical.\end{enumerate}  In case (A), let $s$ be such that $\sh^{s-i_1}(x_{k+i_1})=\tau$.  Then because $j_1$ and $j_2$ are successive flip columns relative to $x_k$, it must be the case that $s>j_2$.  If $s$ is small enough to still be a flip column for $x_k$ then we know that $s-j_2-1\ge m_{t+1}$ by our assumptions that no flip column for $x_k$ satisfies (a) and (b).  If $s$ is so large that it cannot be counted as a flip column relative to $x_k$ then it must be the case that $s>N_\delta$.  So we again have $s-j_2-1\ge m_{t+1}$.  This implies that the first $M$-length word of $\sh^{j_2-i_1+1}(z_1)$ equals the first $M$ word of $\sh^{j_2-i_1+1}(x_{k+i_1})$.  So we have $$\tau_{j_2-j_1+1}\tau_{j_2-j_1+2}\dots\tau_{j_2-j_1+M}= \sh^{j_2-i_1+1}(x_{k+i_1})\rtr_{M}.$$  Cases (B) and (C) lead to the same conclusion via simpler reasoning.

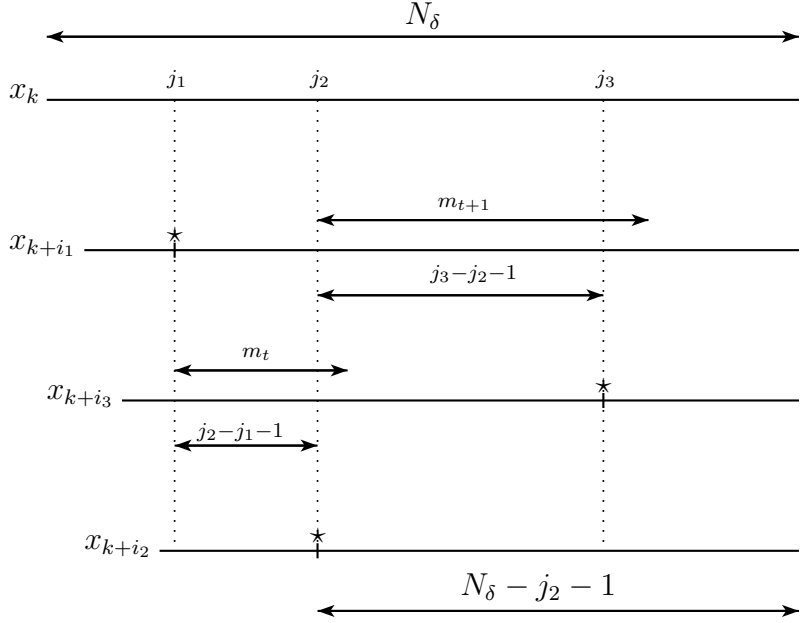
\begin{figure}[ht]
\begin{center}
\psset{xunit=2.0cm,yunit=2.0cm,algebraic=true,dotstyle=o,dotsize=3pt 0,linewidth=0.8pt,arrowsize=3pt 2,arrowinset=0.25}
\begin{pspicture*}(0.5,0.5)(6.5,5)
\psline(1.0,4)(6,4)
\psline(1.25,3)(6,3)
\psline(1.5,2)(6,2)
\psline(1.75,1)(6,1)
\rput[tl](0.75,4.1){\textbf{$x_k$}}
\rput[tl](0.75,3.1){\textbf{$x_{k+i_1}$}}
\rput[tl](1.0,2.1){\textbf{$x_{k+i_3}$}}
\rput[tl](1.25,1.1){\textbf{$x_{k+i_2}$}}
\psline[linestyle=dotted](2.8,4.0)(2.8,1.0)
\psline[linestyle=dotted](4.7,4.0)(4.7,1.0)
\psline[linestyle=dotted](1.85,4)(1.85,1.0)

\psline{->}(2.8,3.2)(5.0,3.2)
\psline{->}(5.0,3.2)(2.8,3.2)

\psline{->}(2.8,1.7)(1.85,1.7)
\psline{->}(1.85,1.7)(2.8,1.7)

\psline{->}(1.0,4.42)(6.0,4.42)
\psline{->}(6.0,4.42)(1.0,4.42)

\psline{->}(2.8,0.6)(6.0,0.6)
\psline{->}(6.0,0.6)(2.8,0.6)

\psline{->}(2.8,2.7)(4.7,2.7)
\psline{->}(4.7,2.7)(2.8,2.7)

\psline{->}(3.0,2.2)(1.85,2.2)
\psline{->}(1.85,2.2)(3.0,2.2)

\rput[tl](3.38,4.65){\textbf{$N_{\delta}$}}
\rput[tl](3.75,0.85){\textbf{$N_{\delta}-j_2-1$}}
\rput[tl](2,1.85){\textbf{$^{j_2-j_1-1}$}}
\rput[tl](2.3,2.35){\textbf{$^{m_t}$}}
\rput[tl](3.6,3.35){\textbf{$^{m_{t+1}}$}}
\rput[tl](1.8,4.2){\textbf{$^{j_1}$}}
\rput[tl](2.75,4.2){\textbf{$^{j_2}$}}
\rput[tl](4.65,4.2){\textbf{$^{j_3}$}}
\rput[tl](3.55,2.9){\textbf{$^{j_3-j_2-1}$}}
\rput[tl](1.8,3.15){$\star$}
\psline(1.85,2.95)(1.85,3.05)
\rput[tl](2.75,1.15){$\star$}
\psline(2.8,0.95)(2.8,1.05)
\rput[tl](4.65,2.15){$\star$}
\psline(4.7,1.95)(4.7,2.05)
\end{pspicture*}
\caption{Return times in the $\delta$ pseudo-orbit of Lemma \ref{nonperiodicsublemma}; `$\star$' represents a flip.}
\end{center}
\end{figure}

  By a similar argument, considering three cases for $z_2$ and $x_{k+i_2}$, we see that $$\tau_1\tau_2\dots \tau_M=\sh^{j_2-i_2+1}(x_{k+i_2}).$$  So by equation (\ref{sublemmaequation}) we see that $$\tau_{j_2-j_1+1}\tau_{j_2-j_1+2}\dots\tau_{j_2-j_1+M}\app \tau_1\tau_2\dots\tau_M.$$  since $\tau$ is non-periodic, none of these symbols can be $*$.  Therefore, $$\tau_{j_2-j_1+1}\tau_{j_2-j_1+2}\dots\tau_{j_2-j_1+M}= \tau_1\tau_2\dots\tau_M.$$  Since $M\ge m_{t+1}$ it then must be the case that $j_2-j_1+1\ge r_{m+1}$, but we assumed that $j_2-j_1-1<m_t$ which implies that $j_2-j_1+1\le m_t+1 \le r_{m_t}+1<m_{t+1}\le r_{m_{t+1}}$, a contradiction.
\end{proof}

We will now define sequences $\{\al_{k}\}_{k\in\N}$, $\{\beta_{k}\}_{k\in\N}$, $\{j_{k}\}_{k\in\N}$, and $\{i_{k}\}_{k\in\N}$, that keep track of the flip columns and rows in a $\delta$ pseudo-orbit as follows.

\begin{defn}\label{def:flipseq}
Let $\eps>0$, and let $N_{\eps}$ be given as in Lemma \ref{brentlemma}.  Let $\delta>0$, and let $\{x_{i}\}_{i\in \N}$ be a $\delta$ pseudo-orbit, with $x_i=x^i_0x^i_1\ldots$ for each $i\in\nat$. Let $\al_{1}$ be minimal such that there exists a $j_{1}<N_{\eps}$ minimal for which either
\begin{enumerate}
	\item there exists $i_{1}\le j_{1}$, minimal, such that $j_{1}$ is a flip column relative to $\al_{1}$ and $i_{1}$ is the flip row of the flip column $j_1$, in other words
	$$x^{\al_{1}}_{j_{1}}\neq x^{\al_{1}+i_{1}}_{j_{1}-i_{1}};$$
	or
	\item we have that
	$$x^{\al_{1}}_{j_{1}}=x^{\al_{1}+i}_{j_{1}-i} = *$$
	for all $1\le i \le j_{1}$ in which case we define $i_{1}=1$.
\end{enumerate}
Let $\beta_{1}=\al_{1}+j_{1}$.  For $n>1$ define $\al_{n}$, $j_{n}$, and $i_{n}$ recursively as above so that $\al_{n}>\beta_{n-1}$. Thus we only start looking for the next point with a flip column in its initial $N_{\eps}$-segment after we have passed the previous flip column.
\end{defn}

These sequences of $\al_k$'s and $\beta_k$'s allow us to keep track of the key flip columns in the $\delta$ pseudo-orbit.  Specifically, in trying to construct a shadowing point for this $\delta$ pseudo-orbit we need to know where the flip columns affect the $N_\eps$ agreement.  Notice that $\sh(x_i)\rtr_{N_\delta}\simeq x_{i+1}\rtr_{N_{\delta}}$, for every $i\in \N$, while our goal is to get a point $z$ which has $\sh^i(z)\rtr_{N_\eps}\simeq x_i\rtr_{N_\eps}$, for all $i\in \N$, with $N_\eps$ much smaller than $N_\delta$.  The $\alpha_k$'s give us the points in the pseudo-orbit that have a flip column somewhere in their first $N_\eps$-many symbols, while the $\beta_k$'s are the actual place where the canonical shadowing point must have a $\diamond$ symbol.

The next result will help us prove that for every $\eps>0$ there is a $\delta_\eps>0$ small enough to guarantee that after the $\diamond$ symbol (in some $\beta_k$th place) in our canonical shadowing point, there is a well-defined string of length $N_\eps$ that agrees with $\tau$. This means that the actual shadowing point will have its $\beta_k$th shift within $\eps$ of the critical point.  This property will allow us to freely choosing $0$ or $1$in place of each $\diamond$, as either choice is allowed in these self-similar dendrite maps.

\begin{prop}\label{canprop}
Let $\eps>0$ and $N_{\eps}\in\N$ be given. Then there is some $\delta_{\eps}>0$ such that for any $0<\delta\le \delta_{\eps}$ and any $\delta$ pseudo-orbit, $\{x_{i}\}_{i\in\N}$ we have
$$
x_{\beta_{k}+t}\rtr_{N_{\eps}-t}\app \sh^{t}(*\tau)\rtr_{N_{\eps}-t}
$$
for all $k\in\nat$ and all $0\le t\le N_{\eps}$.
\end{prop}

\begin{proof}
  We prove the proposition in three parts.  First we assume that $\tau$ is periodic with period $P$, then we assume that $\tau$ is non-recurrent, and finally we assume that $\tau$ is recurrent but non-periodic.

  To begin, let $\tau$ be periodic with period $P\in \N$.  So $$\tau=*\tau_1\tau_2\dots \tau_{P-1}*\tau_1\dots$$  In this case we choose $\delta_\eps>0$ so that $$N_{\delta_\eps}>2(P+1)N_\eps.$$  Choose $0<\delta\le \delta_\eps$ and $N_\delta\in \N$ via Lemma \ref{brentlemma}.  Let $k\in \N$ and consider $\al_k$.  First we consider the case that $\al_k$ satisfies case (1) of Definition \ref{def:flipseq}.  Let flip column $j_k$ and associated flip row $i_k$ be as defined.  We prove that there are no flips relative to $x_{\al_k+i_k}$ between $j_k-i_k+1$ and $(j_k-i_k+1)+N_\eps$ with flip rows between $1$ and $(\beta_k-i_k)+N_\eps$.  Lemma \ref{equalitybetweenflips} will imply the proposition.

  As in Lemma \ref{periodicsublemma}, let $1\le s_1\le (\beta_k-i_k)+N_\eps$ be the minimal flip row for a flip column, $t_1$, relative to $x_{\al_k+i_k}$ where $t_1-j_k$ is not a multiple of $P$.  Given $s_q$ and $t_q$, define $s_q<s_{q+1}\le (\beta_k-i_k)+N_\eps$ to be the least flip row associated with a flip column, $t_{q+1}$, relative to $x_{\al_k+i_k}$ such that $t_{q+1}<t_q$ and such that $t_{q+1}-j_k$ is not a multiple of $P$.  Then $$N_\delta-t_q<qP$$ for all such $q$.  Let $r$ be the number of flip rows and columns relative to $x_{\al_k+i_k}$ enumerated as above.  Then $r$ is obviously bounded by the number of rows between $i_k$ and $j_k-N_\eps$.  This is maximized when $i_k=1$ and $j_k=N_\eps-1$.  In that case we have $r\le N_\eps-1+N_\eps<2N_\eps$.  Therefore the last flip column in our enumeration, $t_r$ satisfies $$N_\delta-t_r<rP<2N_\eps P.$$  We prove that $t_r>j_k+N_\eps$.  Since $$N_\delta>2PN_\eps+2N_\eps,$$ we have $$t_r>N_\delta-rP\ge N_\delta-2N_\eps P>2N_\eps.$$  Since $j_k<N_\eps$, $$t_r>2N_\eps>j_k+N_\eps.$$  The case that $\al_k$ satisfies (2) of Definition \ref{def:flipseq} is more straightforward.  This completes the proof for the periodic case.

  Now assume that $\tau$ is non-recurrent.  Let $M$ be large enough so that $r_m=\infty$ for all $m\ge M$.  That is to say the word $\tau_1\dots \tau_M$ never re-occurs in $\tau$.  Let $\delta_\eps>0$ be chosen so that $$N_{\delta_\eps}>2(N_\eps+M+1).$$  Let $0<\delta\le \delta_\eps$ and let $N_\delta\in \N$ be defined as in Lemma \ref{brentlemma}.  Let $k\in \N$ and consider $\al_k$.  Suppose that $\al_k$ satisfies (1) of Definition \ref{def:flipseq}.  Let $j_k$ and $i_k$ be the flip column and flip row as defined.  We prove that there are no flips relative to $x_{\al_k+i_k}$ with flip columns between $j_k-i_k+1$ and $(j_k-i_k+1)+N_\eps$ and flip rows between $1$ and $(\beta_k-i_k)+N_\eps$.  Let $z$ be the precritical point that is responsible for the $\al_k$ flip.  That is to say
$$
\sh(x_{\al_k+i_k-1})\rtr_{N_\delta}\approx z\rtr_{N_\delta}\approx x_{\al_k+i_k}\rtr_{N_\delta}
$$
with
$$
\sh^{j_k-i_k+1}(z)=*\tau_1\tau_2\dots.
$$
We have
\begin{align*}
\sh^{j_k-i_k+2}(z)\rtr_{N_\delta-N_\eps-2} &= \tau_1\dots \tau_{N_\eps+2M}\dots \tau_{N_\delta-N_\eps-2}\\
 &= \sh^{j_k-i_k+2}(x_{\al_k+i_k})\rtr_{N_\delta-N_\eps-2}
\end{align*}
because $j_k-i_k<N_\eps$.  Suppose that there is some other flip relative to $x_{\al_k+i_k}$ with flip column between $j_k-i_k+1$ and $(j_k-i_k+1)+N_\eps$.  Let $r$ be minimal such that the flip row for this column occurs at $x_{\al_k+i_k+r}$ and let $t$ be chosen so that the precritical point $z'$ responsible for this flip has
$$
\sh(x_{\al_k+i_k+r-1})\rtr_{N_\delta}\approx z'\rtr_{N_\delta}\approx x_{\al_k+i_k+r}\rtr_{N_\delta}
$$
with
$$
\sh^{t}(z)=*\tau_1\tau_2\dots
$$
and $j_k-i_k-r<t\le (j_k-i_k-r)+N_\eps$.  Since $r$ is the first row with a flip, it must be the case that
\begin{align*}
\sh^{t+1}(z')\rtr_{M} &= \tau_1\dots \tau_M\\
 &= \sh^{(j_k-i_k+1)+t-(j_k-i_k-r)}(z)\rtr_M\\
 &=\sh^{t-(j_k-i_k-r)}(\tau)\rtr_M.
\end{align*}
This contradicts the assumption that $\tau_1\dots \tau_M$ never reoccurs in $\tau$.  The case that $\alpha_k$ satisfies (2) of Definition \ref{def:flipseq} is similar (notationally much simpler).  The proof follows for non-recurrent $\tau$.

This will lead to another precritical point $z'$ which is responsible for the flip.  The fact that the flip occurs between $j_k-i_k+1$ and $(j_k-i_k+1)+N_\eps$ implies that there we must have some $t$ such that $\sh^t(z')=*\tau_1\tau_2\dots$ and such

Next assume that $\tau$ is recurrent but not periodic and we have enumerated the return times as in Lemma \ref{nonperiodicsublemma}.  Specifically, let $\{m_i\}_{i\in \N}$ be a sequence of natural numbers chosen such that
$$
m_i\le r_{m_i}<r_{m_i}+1<m_{i+1}.
$$
Let $t$ be such that $m_t>N_\eps$.  Then in this case define $\delta_\eps>0$ so small that
$$
N_{\delta_\eps}>\left(\sum_{q=t}^{t+2N_\eps+1}m_q \right)+2N_\eps+1.
$$
Let $0<\delta\le \delta_\eps$, and let $\{x_i\}_{i\in \N}$ be a $\delta$ pseudo-orbit.  As in the periodic case of the proof, let $k\in \N$ and consider $\al_k$.  First we consider the case that $\al_k$ satisfies case (1) of Definition \ref{def:flipseq}.  Let flip column $j_k$ and associated flip row $i_k$ be as defined.  We prove that there are no flips relative to $x_{\al_k+i_k}$ between $j_k-i_k$ and $j_k-i_k+N_\eps$ with flip rows between $1$ and $\beta_k+N_\eps-i_k$.  Lemma \ref{equalitybetweenflips} will imply the proposition.

Suppose that there is a flip column relative to $x_{\al_k+i_k}$ in column $t_1$ with $j_k-i_k<t_1<j_k-i_k+N_\eps$.  Choose $t_1$ minimally with respect to this property.  Then we have that $t_1-(j_k-i_k)<N_\eps<m_t$.  So by Lemma \ref{nonperiodicsublemma} there next flip column relative to $x_{\al_k+i_k}$ must occur within $m_{t+1}$ of $t_1$.  Letting that flip column be $t_2$, we have
$$
t_2-t_1-1<m_{t+1}.
$$
Therefore,
$$
t_2-t_1+1\le m_{t+1}+1\le r_{m_{t+1}}+1<m_{t+2}\le r_{m_{t+2}}.
$$
So we can iterate the process where in the $q$th step we get $t_q>t_{q-1}$ is minimal such that $t_q$ is a flip column for $x_{\al_k+i_k}$ and
$$
t_q-t_{q-1}<m_{t+q-1}.
$$
This process must terminate at least by the time we have exhausted the possible flip rows between $x_{\al_k+i_k}$ and $x_{\beta_k+N_\eps}=x_{\al_k+j_k+N_\eps}$.  Let $t_r$ be the last such flip column.  Then $r\le j_k+N_\eps-i_k\le j_k+N_\eps<2N_\eps$.  Therefore $t_r-t_{r-1}<m_{t+2N_\eps-1}$.  This implies that
$$
t_r-j_1=t_r-t_{r-1}+t_{r-1}-\cdots +t_1-j_1
$$
and therefore
$$
t_r-j_1<\sum_{q=t}^{t+2N_\eps-1}m_q.
$$
But $N_\delta-j_1>N_\delta-N_\eps>\left(\sum_{q=t}^{t+2N_\eps} m_q\right)+N_\eps$, so
$$
N_\delta-t_r-1>m_{t+2N\eps}+N_\eps-1>m_{r}.
$$
But in this case we can apply Lemma \ref{nonperiodicsublemma} again to construct a flip row $t_{r+1}$ for $x_{\al_k+i_k}$ with $t_{r+1}-t_r<m_r$.  This contradicts the fact that $t_r$ is the last flip column relative to $x_{\al_k+i_k}$.  Again, the case that $\alpha_k$ satisfies (2) of Definition \ref{def:flipseq} is similar.  The proposition follows.
\end{proof}

\begin{prop}\label{prop:pseudo_agreement}
Let $\eps>0$, with $N_{\eps}\in\N$ as in Lemma \ref{brentlemma} and $\delta_{\eps}>0$ as in Proposition \ref{canprop}. If $\{x_n\}_{n\in\nat}$ is a $\delta_{\eps}$ pseudo-orbit, where for each $n$ we write $x_n=x^n_0x^n_1\ldots$, then for each $n\in\nat$
\[
\restr{x_n}{N_{\eps}} \simeq x^n_0x^{n+1}_0\ldots x^{n+N_{\eps}}_0.
\]
\end{prop}
\begin{proof}
Pick $n\in\nat$. If there is no $k\in\nat$ with $n \leq \beta_k \leq N_{\eps}$ we have that
\[
\restr{x_n}{N_{\eps}} = x^n_0x^{n+1}_0\ldots x^{n+N_{\eps}}_0
\]
and we are done. So assume that there is some minimal $\beta_k$ between $n$ and $N_{\eps}$, which represents the first flip column $j_k$ relative to $x_i$, with flip row $i_k$. By Proposition \ref{canprop},
\begin{align}
x_{\beta_{k}+t}\rtr_{N_{\eps}-t}\app \sh^{t}(*\tau)\rtr_{N_{\eps}-t} \label{eqn:pseudo_agreement1}
\end{align}
for all $0\le t\le N_{\eps}$. Moreover, since $j_k$ and $i_k$ are minimal we have that
\begin{align}
x^n_r = x^{n+r}_0 \label{eqn:pseudo_agreement2}
\end{align}
for $0 \leq r < j_k$, and $x^n_{j_k+s} = x^{n+i_k-1}_{j_k-i_k+s+1}$; in other words that
\begin{align}
\restr{\sigma^{j_k}(x_n)}{N_{\eps}-j_k} = \restr{\sigma^{j_k-i_k+1}(x_{n+i_k-1})}{N_{\eps}-j_k}. \label{eqn:pseudo_agreement3}
\end{align}
Since there is a flip in row $i_k$ relative to $x_n$,
\begin{align*}
\restr{\sigma^{j_k-i_k+1}(x_{n+i_k-1})}{N_{\eps}-j_k} &\app \restr{*\tau}{N_{\eps}-j_k}\\
 &\app x_{\beta_{k}}\rtr_{N_{\eps}-j_k}
\end{align*}
by \eqref{eqn:pseudo_agreement1}. The result follows by \eqref{eqn:pseudo_agreement2} and \eqref{eqn:pseudo_agreement3}.
\end{proof}


\section{Shadowing and $\w$-Limit Sets in Dendrites}

In this section we prove that for a $\Lambda$-acceptable $\tau$, the shift map on the dendrite $\C D_\tau$ has shadowing, and we use this fact to prove the main theorem: a closed set $B\subset {\C D_\tau}$ is an $\w$-limit set of a point if, and only if, $B$ is internally chain transitive (ICT).

Recall the definition of the sequences $\{\alpha_k\}_{k\in\nat}$ and $\{\beta_k\}_{k\in\nat}$ from Definition \ref{def:flipseq}.

\begin{defn}\label{def:canonical_shadow}  Let $\eps>0$, let $N_{\eps}\in\nat$ be given by Lemma \ref{brentlemma} and $\delta_{\eps}>0$ be given by Proposition \ref{canprop}. Let $x_1x_2\ldots$ be a $\delta_{\eps}$ pseudo-orbit (finite or infinite), for each $i$ write $x_i=x^i_0x^i_1\ldots$ and define $\hat{z}=\hat{z}_0\hat{z}_1\ldots$ by
\[
\hat{z}_i= \begin{cases} x^i_0 & \mbox{ if }i\neq\beta_k\mbox{ for any }k\in\nat,\\ \diamond & \mbox{ if }i=\beta_k\mbox{ for some }k\in\nat.
\end{cases}
\]
If the pseudo-orbit is finite with last point $x_n$, let $\hat{z}_{n+i}=x^n_i$ for every $i\geq 0$.
The sequence $\hat{z}$ is called a \emph{canonical $\eps$-shadow} for the $\delta_{\eps}$ pseudo-orbit $\{x_i\}$.
\end{defn}



\begin{thm}\label{thm:canonical_shadowing}
Let $\tau$ be $\Lambda$-acceptable.  Then the shift map on the dendrite $\C D_\tau$ has shadowing.
\end{thm}
\begin{proof}
For a given $\eps>0$, let $N_{\eps}$ and $\delta_{\eps}$ be as given in Definition \ref{def:canonical_shadow}, let $\{x_i\}_{i\in\nat}$ be a $\delta_{\eps}$ pseudo-orbit, and let $\hat{z}$ be the canonical $\eps$-shadow for $\{x_i\}_{i\in\nat}$.
Let $z\in \{0,1,*\}^\w$ be such that $z_i=\hat{z}_i$ for all $i\neq \beta_k$.  If for all $k\in \N$, $\sh^{\beta_k}(\hat{z})\not\approx \tau$ let $z_{\beta_k}$ be either $0$ or $1$; if instead there is some least $k$ such that $\sh^{\beta_k}(\hat{z})\approx \tau$ then define $z_{\beta_k}=*$ and for all $i>\beta_k$ let $z_i=\tau_{i-\beta_k}$. Defined in this way, it is clear that $z\in{\C D_{\tau}}$.

We show that $d(\sigma^i(z),x_i)<\eps$ for every $i\in\nat$, which by Lemma \ref{brentlemma} is equivalent to saying that
\[
\restr{\sigma^i(z)}{N_{\eps}}\simeq\restr{x_i}{N_{\eps}}.
\]
Pick $i\in\nat$, then by Proposition \ref{prop:pseudo_agreement}
\[
x_{i}\rtr_{N_{\eps}} \simeq x^i_0x^{i+1}_0\ldots x^{i+N_{\eps}}_0
\]
Thus $\restr{\sigma^i(z)}{N_{\eps}}\simeq\restr{x_i}{N_{\eps}}$ by the definition of the canonical shadow, and so $d(\sigma^i(z),x_i)<\eps$.
\end{proof}

We now prove our main theorem, which is the following.

\begin{thm}\label{thm:can_shad_ICT_property}
Let $\tau$ be $\Lambda$-acceptable.  Then $B\subset {\C D_{\tau}}$ is closed and internally chain transitive if, and only if, $B=\omega(z)$ for some $z \in {\C D_{\tau}}$.
\end{thm}
\begin{proof}
$\C D_{\tau}$ has shadowing by Theorem \ref{thm:canonical_shadowing}. For every $i\in\nat$ let $\delta_i$ be the constant given by shadowing such that each $\delta_i$ pseudo-orbit is $1/2^i$-shadowed.

Suppose that $B\subset {\C D_{\tau}}$ is closed and internally chain transitive. If $B$ is finite, to be ICT it must be a cycle and is thus an $\omega$-limit set, so assume that $B$ is infinite; then to be ICT it must contain infinitely many non-precritical points.

By compactness of $B$ let $\{x_i\}_{i\in \N}$ be a dense subset of $B$, then since $B$ is ICT, for every $i\in\nat$ there is a $\delta_i$ pseudo-orbit $\Gamma_i=\{x_i=x_1^i,\ldots,x^i_{n_i}=x_{i+1}\}$, with $n_i>N_{1/2^{i-1}}$ for $N_{1/2^{i-1}}$ given by Lemma \ref{brentlemma}. (If the $\delta_i$ pseudo-orbit between $x_i$ and $x_{i+1}$ has length less than $N_{1/2^{i-1}}$ we can add extra $\delta_i$ pseudo-orbits from $x_{i+1}$ to itself until the inequality is satisfied.)

For each $i\in\nat$, let $z_i=z_0^iz_1^iz_2^i\dots$ be an assignment of the canonical $1/2^i$-shadow for $\Gamma_i$ as constructed in the proof of Theorem \ref{thm:canonical_shadowing}, except that we do not assign a $*$ to any $\diamond$ in any $z_i$. Define $z\in\{0,1\}^{\omega}$ by
\[
z=z_0^1 z_1^1 \ldots z_{n_1-1}^1 z_0^2 z_1^2 \ldots z_{n_2-1}^2 z_0^3 z_1^3 \ldots
\]
and let $\nu_i=\sum_{j\leq i}n_j$. We claim that $z\in{\C D_{\tau}}$ and that $\omega(z)=B$.

To see that $z\in{\C D_{\tau}}$, notice first that since $B$ contains infinitely many non-precritical points, elements of $\{z_i\}_{i\in\nat}$ will contain arbitrarily long initial segments of non-precritical points, and we deduce that $\sigma^k(z)\neq\tau$ for any $k\in\nat$. If $z\notin{\C D_{\tau}}$, then by the definition of ${\C D_{\tau}}$ we must have that for some $k\in\nat$, $\sigma^k(z)$ only differs from $\tau$ in a place where one has a $*$; $z$ has no $*$ so $\tau$ must have a $*$ in this place, and as such is periodic. But then by the construction of $z$, cofinitely many of the $z_i$ would have to be precritical, forcing $B$ to be the critical cycle, a contradiction.

To see that $B\subset\omega(z)$, notice that $\restr{\sigma^{\nu_i}(z)}{N_{1/2^i}}=\restr{z_i}{N_{1/2^i}}$ for every $i$, so
\[
d(\sigma^{\nu_i}(z),x_i)<\frac{1}{2^i}
\]
for every $i\in\nat$ by the definition of the canonical shadow. For $y\in B$ and $\eps>0$, pick $j\in\nat$ so that $\max\{d(x_j,y),1/2^j\}<\eps/2$. Then
\begin{align*}
d(\sigma^{\nu_j}(z),y) &\leq d(\sigma^{\nu_j}(z),x_j) + d(x_j,y)\\
 &< \frac{1}{2^j} + d(x_j,y)\\
 &< \eps/2 + \eps/2,
\end{align*}
so $y\in\omega(z)$.

To see that $\omega(z)\subset B$, notice first that by the definition of the canonical shadow for finite pseudo-orbits, each $z_i$ begins with a length-$N_{1/2^{i-1}}$ portion of $x_i$ and ends with the whole of $x_{i+1}$, so we immediately see that there is an $N_{1/2^{i-1}}$ overlap between $z_i$ and $z_{i+1}$ for every $i\in\nat$.

Suppose for a contradiction that there is some $y\in\omega(z)\setminus B$. Since $B$ is compact there is some $\eta>0$ such that $y\notin B_{\eta}(B)$; pick $J\in\nat$ such that $1/2^J<\eta$. Since $y\in\omega(z)$, for infinitely many $i>J$ we must have that
\begin{align}
\restr{\sigma^{k_i}(z)}{N_{1/2^i}} &\simeq\restr{y}{N_{1/2^i}},\label{doubledragon}
\end{align}
for appropriate integers $k_i$, by Lemma \ref{brentlemma}. Since $y\notin B$, $\restr{y}{N_{\eta}}\not\simeq\restr{x}{N_{\eta}}$ for all $x\in B$. But then by (\ref{doubledragon}), for infinitely many $i>J$,
\[
\restr{\sigma^{k_i}(z)}{N_{\eta}} \not\simeq \restr{x}{N_{\eta}}
\]
for any $x\in B$, meaning that for infinitely many $z_i$,
\[
\restr{\sigma^{j_i}(z_i)}{N_{\eta}} \not\simeq \restr{x}{N_{\eta}}. \]
This contradicts the fact that the $z_i$ $1/2^i$-shadow points in $B$, so $y\in\omega(z)$ and we are done.
\end{proof}

Finally, Theorems \ref{thm:dendrite_conj} and \ref{thm:can_shad_ICT_property} imply the following:

\begin{col}\label{col:ICT_property}
For any quadratic map $f:J_c\rightarrow J_c$ with a dendritic Julia set $J_c$, $\Lambda\subset J_c$ is closed and internally chain transitive if and only if $\Lambda=\omega(z)$ for some $z\in J_c$.
\end{col}

\bibliographystyle{plain}
\bibliography{Bibtex_dendrites}

\def\ocirc#1{\ifmmode\setbox0=\hbox{$#1$}\dimen0=\ht0 \advance\dimen0
  by1pt\rlap{\hbox to\wd0{\hss\raise\dimen0
  \hbox{\hskip.2em$\scriptscriptstyle\circ$}\hss}}#1\else {\accent"17 #1}\fi}
  \def\ocirc#1{\ifmmode\setbox0=\hbox{$#1$}\dimen0=\ht0 \advance\dimen0
  by1pt\rlap{\hbox to\wd0{\hss\raise\dimen0
  \hbox{\hskip.2em$\scriptscriptstyle\circ$}\hss}}#1\else {\accent"17 #1}\fi}
  \def\cprime{$'$}
\begin{thebibliography}{10}

\bibitem{AlsedaFagella}
L.~Alsed{\`a} and N.~Fagella.
\newblock Dynamics on {H}ubbard trees.
\newblock {\em Fund. Math.}, 164(2):115--141, 2000.

\bibitem{Baldwin-JFPTA}
S.~Baldwin.
\newblock Continuous itinerary functions and dendrite maps.
\newblock {\em Topology Appl.}, 154(16):2889--2938, 2007.

\bibitem{Baldwin-Fund}
S.~Baldwin.
\newblock Inverse limits of tentlike maps on trees.
\newblock {\em Fund. Math.}, 207(3):211--254, 2010.

\bibitem{Baldwin-TopApp}
S.~Baldwin.
\newblock Julia sets and periodic kneading sequences.
\newblock {\em J. Fixed Point Theory Appl.}, 7(1):201--222, 2010.

\bibitem{BGKR-omega}
A.~Barwell, C.~Good, R.~Knight, and B.~E. Raines.
\newblock A characterization of {$\omega$}-limit sets in shift spaces.
\newblock {\em Ergodic Theory Dynam. Systems}, 30(1):21--31, 2010.

\bibitem{Barwell-Fund}
A.~D. Barwell.
\newblock A characterization of {$\omega$}-limit sets of piecewise monotone
  maps of the interval.
\newblock {\em Fund. Math.}, 207(2):161--174, 2010.

\bibitem{BDG-tentmaps}
A.~D. Barwell, C.~Good, and G.~Davies.
\newblock On the $\omega$-limit sets of tent maps.
\newblock {\em Fund. Math.}, (to appear).

\bibitem{BGOR-preprint}
A.~D. Barwell, C.~Good, P.~Oprocha, and B.~E. Raines.
\newblock Characterizations of $\omega$-limit sets of topologically hyperbolic
  spaces.
\newblock {\em arXiv}, (1110.3219), 2011.

\bibitem{D-H}
A.~Douady and J.~H. Hubbard.
\newblock {\em \'{E}tude dynamique des polyn\^omes complexes. {P}artie {I}},
  volume~84 of {\em Publications Math\'ematiques d'Orsay [Mathematical
  Publications of Orsay]}.
\newblock Universit\'e de Paris-Sud, D\'epartement de Math\'ematiques, Orsay,
  1984.

\bibitem{Hirsch}
M.~W. Hirsch, Hal~L. Smith, and X.-Q. Zhao.
\newblock Chain transitivity, attractivity, and strong repellors for
  semidynamical systems.
\newblock {\em J. Dynam. Differential Equations}, 13(1):107--131, 2001.

\bibitem{Schleicher}
D.~Schleicher.
\newblock On fibers and local connectivity of {M}andelbrot and {M}ultibrot
  sets.
\newblock In {\em Fractal geometry and applications: a jubilee of {B}eno\^\i t
  {M}andelbrot. {P}art 1}, volume~72 of {\em Proc. Sympos. Pure Math.}, pages
  477--517. Amer. Math. Soc., Providence, RI, 2004.

\bibitem{Walters}
P.~Walters.
\newblock On the pseudo-orbit tracing property and its relationship to
  stability.
\newblock In {\em The structure of attractors in dynamical systems ({P}roc.
  {C}onf., {N}orth {D}akota {S}tate {U}niv., {F}argo, {N}.{D}., 1977)}, volume
  668 of {\em Lecture Notes in Math.}, pages 231--244. Springer, Berlin, 1978.

\end{thebibliography}

\end{document}